\documentclass[12pt]{amsart}
\usepackage{amssymb}

\numberwithin{equation}{section}
\theoremstyle{plain}
\newtheorem{thm}{Theorem}[section]
\newtheorem{prop}[thm]{Proposition}
\newtheorem{cor}[thm]{Corollary}
\newtheorem{lem}[thm]{Lemma}

\theoremstyle{definition}

\theoremstyle{remark}

\def\R{{\mathbb R}}
\def\N{{\mathbb N}}

\def\C{{\mathbb C}}

\def\p#1{{\left({#1}\right)}}
\def\b#1{{\left\{{#1}\right\}}}

\def\n#1{{\left\|{#1}\right\|}}
\def\abs#1{{\left|{#1}\right|}}
\def\jp#1{{\left\langle{#1}\right\rangle}}

\def\supp{\operatorname{supp}}

\title{Structural resolvent estimates and \\
derivative nonlinear Schr\"odinger
 equations}

\author[]{Michael Ruzhansky and Mitsuru Sugimoto}
\address{
  Michael Ruzhansky:
  \endgraf
  Department of Mathematics
  \endgraf
  Imperial College London
  \endgraf
  180 Queen's Gate, London SW7 2AZ, UK
  \endgraf
  {\it E-mail address} {\rm m.ruzhansky@imperial.ac.uk}
  \endgraf
  \medskip
  Mitsuru Sugimoto:
  \endgraf
  Graduate School of Mathematics
  \endgraf
  Nagoya University
  \endgraf
  Furocho, Chikusa-ku, Nagoya 464-8602, Japan
  \endgraf
  {\it E-mail address} {\rm sugimoto@math.nagoya-u.ac.jp}
  }

\thanks{The first
 author was supported by the EPSRC
 Leadership Fellowship EP/G007233/1.
}

\date{\today}

\begin{document}
\maketitle
\begin{abstract}
A refinement of uniform resolvent estimate is given
and several smoothing estimates for Schr\"odinger equations
in the critical case are induced from it.
The relation between this resolvent estimate and radiation
condition is discussed.
As an application of critical smoothing estimates, 
we show a global existence results 
for derivative nonlinear Schr\"odinger equations.
\end{abstract}
\section{Introduction}
Let us consider the resolvent operator
\[
R(z)=(-\Delta-z)^{-1}
\]
on $\R^n$.
It is defined only for $z\in\C\setminus\b{x\in\R;x>0}$
as an element of $\mathcal{L}(L^2,L^2)$,
but the weak limit
\[
R(\lambda\pm i0)=\lim_{\varepsilon\searrow0} R(\lambda\pm i\varepsilon)
\]
exists in $\mathcal{L}(L^2_k,L^2_{-k})$ for $k>1/2$, where
$L^2_m$ ($m\in\R$) denotes the set of 
functions $g$ such that the norm
\[
\n{g}_{L^2_m}
=\p{\int_{\R^n}\abs{\langle x\rangle^m g(x)}^2\,dx}^{1/2}
;\quad
\jp{x}=\p{1+|x|^2}^{1/2}
\]
is finite.
This fact was first pointed out by Agmon \cite{A}, and
is called {\it the limiting absorption principle}.
This principle can be justified by the resolvent estimate
\[
\sup_{\varepsilon>0}\n{R(\lambda\pm i\varepsilon)v}_{L^2_{-k}}
\le C_\lambda
\n{v}_{L^2_k}
\]
for $\lambda>0$ and $k>1/2$.
More generally, we have 
\begin{equation}\label{Agmon}
\n{D^\alpha R(\lambda\pm i0)v}_{L^2_{-k}}
\le C_\lambda
\n{v}_{L^2_k}
\end{equation}
for $\lambda\neq0$, $|\alpha|\leq2$, and $k>1/2$.
\par
Furthermore, if $n\geq2$, we have the uniform resolvent estimates
\begin{equation}\label{1.1}
\sup_{\lambda\in\R}\n{|D| R(\lambda\pm i0)v}_{L^2_{-k}}
\le C
\n{v}_{L^2_k},
\end{equation}
or equivalently
\begin{equation}\label{1.2}
\sup_{\lambda\in\R}
\n{
\sigma(X,D)
R(\lambda\pm i0)
\sigma(X,D)^*
v}_{L^2}
\le
C\n{v}_{L^2}
\end{equation}
for $\sigma(x,\xi)=\jp{x}^{-k}|\xi|^{1/2}$,
where $k>1/2$.
The uniform resolvent estimates such as \eqref{1.1} and \eqref{1.2}
are used to show
smoothing estimates for Schr\"odinger equations.
See, for example, \cite{KY}, \cite{W}, \cite{Ho1}, \cite{Ho2},
\cite{Su1}, \cite{Su2}, \cite{SuT}, \cite{RS2}.
We remark that we have used here the notation
\[
 \sigma(X,D,Y)
 =\p{2\pi}^{-n}\int_{\R^n}\int_{\R^n}
   e^{i(x-y)\cdot \xi}\sigma(x,y,\xi)u(y) dy d\xi
\]
as a pseudo-differential operator following Kumano-go \cite{Ku}.
We usually abbreviate $X$ (resp.~$Y$) 
if $\sigma(x,y,\xi)$ does not depend on
$x$ (resp.~$y$).
For $\sigma(X,D)$, we set $\sigma(X,D)^*=\bar\sigma(Y,D)$.

\par
The objective of this paper is to establish the following:
\begin{itemize}
\item
In uniform resolvent estimates \eqref{1.1} and \eqref{1.2},
the critical case $k=1/2$ or more general combination of the order
for the weight
can be attained if we assume a structure on $\sigma(X,D)$
(Section \ref{Section3}).
\item
The structure is related to a radiation condition
(Section \ref{Section4}).
\item
Such consideration can be used for the nonlinear problems
for Schr\"odinger equations
(Section \ref{Section6}).
\end{itemize}

Below we give the details,
together with the organisation of this article.
In Section \ref{Section2},
by following the argument in authors' previous work \cite{RS2},
we will show a refined version of the resolvent estimate
\eqref{Agmon}, which is associated with a structure
induced by $-\Delta$.
To understand the geometric meaning of this structure,
we will also consider more general elliptic operators $L$
instead of $-\triangle$.
After such preliminary results,
we will show in Section \ref{Section3} a uniform resolvent estimate
(Theorem \ref{RE}) which includes the 
following result as a special case
\medskip
\begin{thm}\label{RE:Lap}
Let $n\ge2$ and let $\theta\in\R$.
Suppose that $\sigma(x,\xi)\sim|\xi|^{\theta}$,
$\tau(x,\xi)\sim|\xi|^{1-\theta}$,
and $\sigma(x,\xi)=0$ on
$\Gamma=\b{(x,\xi)\in \R^n\times(\R^n\setminus0)\,;\,x\wedge\xi=0}$.
Then we have the estimate
\[
\sup_{\lambda\in\R}
\n{
\sigma(X,D)
R(\lambda\pm i0)\tau(X,D)^*
v}_{L^2_{-1/2+l}(\R^n)}
\le
C\n{v}_{L^2_{1/2+l}(\R^n)}
\]
$0<l<\min\b{1,(n-1)/2}$.
Suppose also that $\tau(x,\xi)=0$ on $\Gamma$.
Then estimate \eqref{eq:res3} is true for
$|l|<\min\b{1,(n-1)/2}$.
\end{thm}
\medskip
Here the notation $\sigma(x,\xi)\sim|\xi|^b$ means that
$\sigma(x,\xi)$ is positively homogeneous of order $b$ in
the variable $\xi$ and its derivatives satisfy a natural decaying
property. 
For the precise definition, see \eqref{sigma'}.
Here we have also used the notation 
$a\wedge b=(a_ib_j-a_jb_i)_{i<j}$ for vectors
$a=(a_1,a_2,\ldots,a_n)$ and $b=(b_1,b_2,\ldots,b_n)$.
Theorem \ref{RE:Lap} with $l=0$ corresponds to estimates 
\eqref{1.1} and \eqref{1.2} with the critical case $k=1/2$, and
furthermore, we have freedom in the choice of $l$.
Such advantages come from the 
extra structure conditions $\sigma(x,\xi)=0$,
$\tau(x,\xi)=0$ on the set $\Gamma$.
We remark that the case of homogeneous weight $|x|^{-1/2}$ 
instead of $\jp{x}^{-1/2}$ was considered in \cite{RS2} (when $l=0$).
\par
In Section 4, we will explain the relation between our result
Theorem \ref{RE:Lap}
and Herbst-Skibsted's resolvent estimate in \cite{HS},
where the symbol $\sigma(x,\xi)=\xi \mp (x/|x|)|\xi|$ 
vanishing only on the half part of $\Gamma$
is used to show results similar to Theorem \ref{RE:Lap}.
We remark that such result induces Sommerfeld's radiation condition.
\par
Theorem \ref{RE:Lap} implies 
many estimates for Schr\"odinger equations.
For example, if we follow the terminology in \cite{KY}, Theorem \ref{RE:Lap}
with the case $\sigma(x,\xi)=\tau(x,\xi)$ and $l=0$
means that
the operator $\tilde\sigma(X,D)=\jp{X}^{-1/2}\sigma(X,D)$ is
$-\Delta$-supersmooth.
Due to Kato's work \cite{K}, this property automatically induces
smoothing estimates for Schr\"odinger equations,
which covers the critical case of the estimates obtained by
Ben-Artzi and Klainerman \cite{BK} or Chihara \cite{Ch2}.
Such results will be displayed in Section \ref{Section5}.
\par
As an application of smoothing estimates induced from Theorem \ref{RE:Lap},
we will consider
in Section \ref{Section6} the existence of time global solutions
to the following derivative nonlinear Schr\"odinger equation
\[
\left\{
\begin{aligned}
\p{i\partial_t+\triangle}\,u(t,x)=&f\p{\nabla u(t,x)}\\
u(0,x)=\varphi(x),\quad &t\in\R,\, x\in\R^n,
\end{aligned}
\right.
\]
where $\triangle$ and $\nabla$ denote the Laplacian 
and the gradient
in $x\in\R^n$, respectively.
The question is which conditions on the 
initial data $\varphi$ guarantee the global in time 
existence of solutions?
Some answers in the case when
$f(u)$ has a polynomial growth
of order $N$ are known:
\begin{itemize}
\item
$\varphi\in C^\infty$, rapidly decay, and sufficiently small
in the case $N\geq3$ (Chihara \cite{Ch}).
\item
$\varphi\in H^{[n/2]+5}$, rapidly decay, and sufficiently small
in the case $N\geq2$
(Hayashi, Miao, and Naumkin \cite{HMN}).
\item
$\varphi\in H^{n/2+2}$, sufficiently small
in the case $N\geq3$
(Ozawa and Zhai \cite{OZ}).

\end{itemize}
\par
By using smoothing estimates obtained in Section \ref{Section5},
we can weaken the smoothness assumption on $\varphi$
if the nonlinear term has a null-form structure like
$f(x/\jp{x}\wedge\nabla u)$.
As it was announced in \cite{RS3}, the regularity index $[n/2]+5$,
or even $n/2+2$, can be replaced by a smaller one in this case.
We give an example of these results:
\medskip
\begin{thm}\label{thm:thm0}
Let $n\geq3$, let $N\geq4$, $N\in\N$, let $\kappa>3(N+1)/4(N-1)$,
and let $\tau>(n+3)/2$.
Assume that $\jp{x}^\kappa \jp{D}^\tau\varphi\in L^2$ and
its $L^2$-norm is sufficiently small.
Then the equation
\[
\left\{
\begin{aligned}
\p{i\partial_t+\triangle}\,u(t,x)=&\abs{\p{x/\jp{x}}
\wedge\nabla u}^N, \\
u(0,x)=\varphi(x),\quad &t\in\R,\, x\in\R^n,
\end{aligned}
\right.
\]
has a time global solution $u\in C^0(\R_t\times\R^n_x)$.
\end{thm}
\medskip
The proof of Theorem \ref{thm:thm0} is a 
simple application of the fixed
point theorem on contraction mapping.
The freedom of the choice of $l$ in 
Theorem \ref{RE:Lap}, which is due to
the structure of nonlinear term,
enables us to induce the contraction directly.
\par
The phenomena that some structure of the nonlinear term
has effects on the regularity problem can be seen in many nonlinear
equations.
For example, Klainerman-Machedon \cite{KM1}, \cite{KM2} showed that
wave equations with nonlinear terms satisfying the {\it null condition}
have local existence and uniqueness in the Sobolev space $H^s$
for smaller $s$ than that of wave equations 
with general nonlinear terms.
Theorem \ref{thm:thm0} can be regarded as one of such phenomena
for Schr\"odinger equations.

\medskip
\par
\section{A resolvent estimate with structure}\label{Section2}
In this section, we establish
a refined version of the resolvent estimate
\eqref{1.2} in Introduction, which is associated with a structure
induced by $-\Delta$.
To describe it, we also generalise the operator
$L=-\triangle$, which will also enable us to clarify its
geometric meaning.
For the purpose, we introduce notations which will
be used in the rest of this paper.
Let
\begin{equation}\label{def:Lp}
\begin{aligned}
 L=p(D)^2\,;\quad &p(\xi)\in C^\infty(\R^n\setminus0),
\quad p(\xi)>0,\quad p(\lambda\xi)=\lambda p(\xi) \quad(\lambda>0),
\\
&\text{$\b{\xi\in\R^n\setminus0: p(\xi)=1}$ has non-vanishing
 Gaussian curvature.}
\end{aligned}
\end{equation} 
Then 
$
\psi(\xi)=p(\xi)\frac{\nabla p(\xi)}{\abs{\nabla p(\xi)}}
$
defines a $C^\infty$-diffeomorphism on $\R^n\setminus0$, and we set
\begin{equation}\label{Omega}
\begin{aligned}
&\Omega(x,\xi)=x\psi'(\xi)^{-1} \wedge\psi(\xi)
=\p{\Omega_{ij}(x,\xi)}_{i<j},
\\
&\omega(x,\xi)=\jp{x}^{-1}\Omega(x,\xi)=\p{\omega_{ij}(x,\xi)}_{i<j}.
\end{aligned}
\end{equation}
We remark that $L=-\triangle$ and 
$\omega(x,\xi)=\p{x/\jp{x}}\wedge \xi$
if $p(\xi)=|\xi|$.
We also set
\begin{align*}
&R(z)=(L-z)^{-1}
=\mathcal F^{-1}_\xi(p(\xi)^2-z)^{-1}\mathcal F_x,
\\
&R(\lambda\pm i0)=\lim_{\varepsilon\searrow0} R(\lambda\pm i\varepsilon)
\end{align*}
for $z\in \C\setminus\b{x\in\R\,;\,x\geq 0}$ and $\lambda\in\R$,
where $\mathcal F_x,\mathcal F^{-1}_\xi$ 
denote the Fourier transformation and its inverse,
respectively, and the limit is taken in the sense of distributions.
For $\lambda<0$, we have $R(\lambda\pm i0)=R(\lambda)$.
We remark that $\Omega$ commutes with functions of the operator $p(D)$.
\medskip
\begin{lem}\label{Com}
Let $h\in C^\infty\p{\R\setminus0}$.
Then the operators of the form
$\p{h\circ p}(D)$ commute with $\Omega_{ij}(X,D)$
whenever it makes sense.
In particular, $R(\lambda\pm i0)$ commutes with $\Omega_{ij}(X,D)$.
\end{lem}
\begin{proof}
See the proof of \cite[Lemma 3.2]{RS2}.
\end{proof}
\medskip
Let us also introduce
the classical orbit $\b{\p{x(t),\xi(t)}:\,t\in\R}$
associated to the operator $L$ defined by \eqref{def:Lp},
which satisfies
\begin{equation}\label{eq:cl}
\left\{
\begin{aligned}
\dot{x}(t)&=\nabla_\xi p^2(\xi(t)), \quad\dot{\xi}(t)=0,
\\
x(0)&=0,\quad \xi(0)=\eta,
\end{aligned}
\right.
\end{equation}
and define the set of the paths of all classical orbits:
\begin{equation}\label{eq:orbits}
\begin{aligned}
\Gamma
=&\b{\p{x(t),\xi(t)}\,:\,t\in\R,
\eta \in\R^n\setminus 0}
\\
=&\b{\p{\lambda\nabla p(\xi),\xi}\,:
\,\xi\in\R^n\setminus 0,\lambda\in\R}.
\end{aligned}
\end{equation}
In the case $L=-\triangle$, for example, we have
\begin{align*}
\Gamma=&\b{\p{\lambda\xi,\xi}\,:
\,\xi\in\R^n\setminus 0,\lambda\in\R}
\\
=&\b{\p{x,\xi}\in\R^n\times(\R^n\setminus0)\,:\,x\wedge\xi=0}.
\end{align*}
By using it, we define a structure induced by the operator
$L$:
\begin{equation}\label{struc}
\sigma(x,\xi)=0\quad \text{if}\quad
(x,\xi)\in\Gamma.
\end{equation}
We remark that $\omega_{ij}(x,\xi)$ in \eqref{Omega} satisfy
\eqref{struc}
(see \cite[Lemma 3.1]{RS2}).
With the dual function $p^*(x)$ of $p(x)$ 
which satisfies $p^*(\nabla p(x))=1$ (see \cite[Theorem 3.1]{RS2}),
elements of
\[
\omega^*(x,\xi)=a(x)\nabla p^*(x)\wedge \xi
\]
are also examples of $\sigma(x,\xi)$
which satisfies \eqref{struc}, where $a(x)$ is an appropriate 
function whose support is away from the origin.

We say that $\sigma(x,\xi)$ is of the class $\mathcal{A}^m_k$
if it satisfies
\[
\abs{\partial_x^\alpha\partial_\xi^\gamma
\sigma(x,\xi)} \leq C_{\alpha\gamma}
\jp{x}^{m-|\alpha|}\jp{\xi}^{k-|\gamma|},
\]
for all $\alpha$ and $\gamma$.
In the case $k=0$, we abbreviate it writing $\mathcal{A}^m$.
Then we have the following proposition:
\medskip
\begin{prop}\label{LAP}
Let $n\geq2$, let $\lambda\in\R$,
and let $\chi\in C^\infty_0\p{\R^n\setminus0}$.
Suppose that $\sigma(x,\xi)\in\mathcal{A}^{-1/2+l}$,
$\tau(x,\xi)\in\mathcal{A}^{-1/2-l}$
and
$\sigma(x,\xi)=0$ on $\Gamma$.
Then there is some $N\in\N$ such that we have the estimate
\begin{align*}
&\n{
\sigma(X,D)
R(\lambda\pm i0)
\chi\p{D}
\tau(X,D)^*
v}_{L^2(\R^n)}
\\
\le
&C_{\lambda,\chi}
\p{
\sup_{\substack{x,\xi\in\R^n\\ |\alpha|+|\gamma|\leq N}}
\abs{
\frac{\partial_x^\alpha\partial_\xi^\gamma\sigma(x,\xi)}
{\jp{x}^{-1/2+l-|\alpha|}\jp{\xi}^{-|\gamma|}}
}
}
\p{
\sup_{\substack{x,\xi\in\R^n\\ |\alpha|+|\gamma|\leq N}}
\abs{
\frac{\partial_x^\alpha\partial_\xi^\gamma\tau(x,\xi)}
{\jp{x}^{-1/2-l-|\alpha|}\jp{\xi}^{-|\gamma|}}
}
}
\n{v}_{L^2(\R^n)}
\end{align*}
for $0<l<1$.
Suppose also that
$\tau(x,\xi)=0$ on $\Gamma$.
Then the estimate is true for $-1<l<1$.
\end{prop}
\medskip
We will now prove Proposition \ref{LAP}.
The proof is a modified version of the arguments in
\cite[Theorem 3.1]{Su2} and \cite[Theorem 4.1]{RS2} and
it may include the repetition of them.
The following lemma
due to \cite[Proposition 3.3]{RS2} (and its proof) is essential:
\medskip
\begin{lem}\label{basiclem}
Let $m\in\R$, let $\varepsilon>0$, and
let $\sigma(x,\xi)\in \mathcal{A}^m_1$.
Suppose that $\sigma(x,\xi)=0$ if
$(x,\xi)\in\Gamma$ or $|\xi|<\varepsilon$.
Then we have
\begin{align*}
&\n{\sigma(X,D)u}_{L^2(\R^n)}
\\
\leq &C
\p{
\sup_{\substack{x,\xi\in\R^n\\ |\alpha|+|\gamma|\leq N}}
\abs{
\frac{\partial_x^\alpha\partial_\xi^\gamma\sigma(x,\xi)}
{\jp{x}^{m-|\alpha|}\jp{\xi}^{1-|\gamma|}}
}
}
\p{
\sum_{i<j}\n{\Omega_{ij}(X,D)u}_{L^2_{m-1}(\R^n)}
+\n{u}_{L^2_{m-1}(\R^n)}
},
\end{align*}
with a sufficiently large $N\in\N$.
\end{lem}
\medskip
We set
\[
K_{\lambda,\chi}=
R(\lambda\pm i0)
\chi\p{D}.
\]
By Lemma \ref{basiclem} and by taking the adjoint, the estimate 
of Proposition \ref{LAP} is reduced to
show the
$L^2(\R^n)$-boundedness of the operator
\[
\widetilde{K}_{\lambda,\chi}
=\jp{x}^{-3/2+l}\p{\Omega_{ij}}^k
 K_{\lambda,\chi}
\jp{x}^{-1/2-l},
\]
for $0<l<1$, and
\[
\widetilde{K}_{\lambda,\chi}
=\jp{x}^{-3/2+l}\p{\Omega_{ij}}^k
 K_{\lambda,\chi}
\p{\Omega_{i'j'}^*}^{k'}\jp{x}^{-3/2-l},
\]
for $-1<l<1$,
where $\Omega_{i'j'}^*$ is the adjoint of $\Omega_{i'j'}$,
and $k,k'=0,1$.

Here we have also used the boundedness of $\tau(X,D)$ from
$L^2_{-1/2-l}$ to $L^2$, which is justified by
\cite[Theorem 3.1]{RS1}, \cite[Theorem 2.1]{RS2}.
In any case, since $\Omega_{ij}$ almost commutes
with $K_{\lambda,\chi}$ by Lemma \ref{Com},
$\widetilde{K}_{\lambda,\chi}$ has the expressions
\begin{align*}
\widetilde{K}_{\lambda,\chi}
=&\sum_{\nu:finite}
  \jp{x}^{-3/2+l}R(\lambda\pm i0)\chi_\nu(D)f_\nu(x)\jp{x}^{-3/2-l}
\\
=&\sum_{\nu:finite}
    \jp{x}^{-3/2+l}\tilde{f}_\nu(x)R(\lambda\pm i0)
    \tilde{\chi}_\nu(D)\jp{x}^{-3/2-l}
\notag
\end{align*}
where $f_\nu,\tilde{f}_\nu$ are functions of polynomial growth of
order $2$ at most, and
$\chi_\nu,\tilde{\chi}_\nu\in C_0^\infty$ have their support in
that of $\chi$.
Hence we may assume
\begin{equation}\label{T}
\widetilde{K}_{\lambda,\chi}
=\jp{x}^{-3/2+l}K_{\lambda,\chi}\jp{x}^{1/2-l},
\qquad
\widetilde{K}_{\lambda,\chi}
=\jp{x}^{1/2+l}K_{\lambda,\chi}\jp{x}^{-3/2-l},
\end{equation}
whichever we need to show the $L^2$ boundedness for $-1<l<1$.
\par
We may assume, as well, that $\chi(\xi)\in C^\infty_0(\R^n)$
has its support in a sufficiently small
conic neighbourhood of $(0,\ldots,0,1)$.
We split the variables in $\R^n$ in the way of
\[
x=(x',x_n),\, x'=(x_1,\ldots,x_{n-1}).
\]
By the integral kernel representation,
we express the operators
$\widetilde{K}_{\lambda,\chi}$ and $K_{\lambda,\chi}$ as
\begin{align*}
&\widetilde{K}_{\lambda,\chi}v(x)=\int \widetilde{K}_{\lambda,\chi}(x,y)v(y)
\,dy
=\int\,dy_n\cdot\int \widetilde{K}_{\lambda,\chi}(x,y)v(y)\,dy',
\\
&K_{\lambda,\chi} v(x)=\int K_{\lambda,\chi}(x,y)v(y)\,dy
=\int\,dy_n\cdot\int K_{\lambda,\chi}(x,y)v(y)\,dy'.
\end{align*}
The following is fundamental in the proof of the limiting absorption principle
(see \cite[Lemma 4.1]{RS2}):
\medskip
\begin{lem}\label{Res}
Let $\chi(\xi)\in C_0^\infty(\R^n)$ have its support in a small
conic neighbourhood of $(0,\ldots,0,1)$.
Then we have
\[
   \n{\int K_{\lambda,\chi}(x,y)v(y)\,dy'}_{L^2(\R_{x'}^{n-1})}
  \leq C_{\lambda,\chi}
  \n{v(\cdot,y_n)}_{L^2(\R^{n-1})},
\]
where $C_{\lambda,\chi}$ is independent of $x_n$ and $y_n$.
\end{lem}
\medskip
By (\ref{T}) and Lemma \ref{Res}, we have easily
\[
   \n{\int\jp{x}^{3/2-l} \widetilde{K}_{\lambda,\chi}(x,y)
     \jp{y}^{-1/2+l}v(y)\,dy'}_{L^2(\R_{x'}^{n-1})}
  \leq C_{\lambda,\chi}
  {\n{v(\cdot,y_n)}_{L^2(\R^{n-1})}}
\]
and
\[
   \n{\int\jp{x}^{-1/2-l} \widetilde{K}_{\lambda,\chi}(x,y)
     \jp{y}^{3/2+l}v(y)\,dy'}_{L^2(\R_{x'}^{n-1})}
  \leq C_{\lambda,\chi}
  {\n{v(\cdot,y_n)}_{L^2(\R^{n-1})}}.
\]
By interpolation, we have
\[
   \n{\int\jp{x}^{1/2\pm\epsilon} \widetilde{K}_{\lambda,\chi}(x,y)
     \jp{y}^{1/2\mp\epsilon}v(y)\,dy'}_{L^2(\R_{x'}^{n-1})}
  \leq C_{\lambda,\chi}
  {\n{v(\cdot,y_n)}_{L^2(\R^{n-1})}},
\]
hence
\[
   \n{\int \widetilde{K}_{\lambda,\chi}(x,y)v(y)\,dy'}_{L^2(\R_{x'}^{n-1})}
  \leq C_{\lambda,\chi}
  \frac{\n{v(\cdot,y_n)}_{L^2(\R^{n-1})}}
   {|x_n|^{1/2\pm\epsilon}|y_n|^{1/2\mp\epsilon}}
\]
for $0<\epsilon\leq1-|l|$.
Since $|x_n|^{-2\epsilon}\leq2^{2\epsilon}|x_n-y_n|^{-2\epsilon}$
if $|x_n|\geq|y_n|$ and
$|y_n|^{-2\epsilon}\leq2^{2\epsilon}|x_n-y_n|^{-2\epsilon}$
if $|x_n|\leq|y_n|$,
we have
\[
   \n{\int \widetilde{K}_{\lambda,\chi}(x,y)v(y)\,dy'}_{L^2(\R_{x'}^{n-1})}
  \leq C_{\lambda,\chi}
  \frac{\n{v(\cdot,y_n)}_{L^2(\R^{n-1})}}
   {|x_n|^{1/2-\epsilon}|x_n-y_n|^{2\epsilon}|y_n|^{1/2-\epsilon}}.
\]
If we take $\epsilon$ such that $0<\epsilon<\min\b{1/2,1-|l|}$, then we have
\begin{align*}
\n{\widetilde{K}_{\lambda,\chi}v}_{L^2(\R^n)}
 &\leq
  \n{
     \int\n{\int \widetilde{K}
      _{\lambda,\chi}(x,y)v(y)\,dy'}_{L^2(\R_{x'}^{n-1})}\,dy_n
    }_{L^2(\R_{x_n})}
\\
 &\leq
  C_{\lambda,\chi}
  \n{
     \int
        \frac{\n{v(\cdot,y_n)}_{L^2(\R^{n-1})}}
         {|x_n|^{1/2-\epsilon}|x_n-y_n|^{2\epsilon}|y_n|^{1/2-\epsilon}}
     \,dy_n
    }_{L^2(\R_{x_n})}
\\
 &\leq
  C_{\lambda,\chi}\n{v}_{L^2(\R^n)},
\end{align*}
where we have used
the following fact (with the case $n=1$)
proved by Hardy-Littlewood \cite{HL}:
\medskip
\begin{lem}\label{SW}
Let $\gamma<n/2$, $\delta <n/2$,
$m<n$, and
$\gamma+\delta+m=n$.
Then we have
\begin{align*}
\p{
\int_{\R^n}
\abs{
\abs{x}^{-\gamma}\abs{D}^{m-n}\abs{x}^{-\delta} f(x)
}^2\,dx
}^{1/2}
&=
\p{
\int_{\R^n}
\abs{
\int_{\R^n}
\dfrac{f\p{y}}{\abs{x}^{\gamma}\abs{x-y}^{m}\abs{y}^\delta}
\,dy
}^2\,dx
}^{1/2}
\\
&\le C
\p{ \int_{\R^n}\abs{f(x)}^2\,dx
}^{1/2}.
\end{align*}
\end{lem}
\medskip
\noindent
(See also \cite[Theorem B]{SW}.)
Thus we have completed the proof of Proposition \ref{LAP}.
\par

\medskip
\par
\section{Uniform resolvent estimates}\label{Section3}
In this section we derive uniform resolvent estimates.
We use the notation
\begin{equation}\label{sigma'}
\sigma(x,\xi)\sim |\xi|^b\qquad (b\in\R)
\end{equation}
in the sense that it satisfies
$\sigma(x,\xi)\in C^\infty\p{\R^n_x\times(\R^n_{\xi}\setminus0)}$
and 
\[
\sigma(x,\lambda\xi)=\lambda^b\sigma(x,\xi)
\,;\,(\lambda>0,\,\xi\neq0),
\quad
\abs{\partial^\alpha_x \sigma(x,\xi)}
\leq C_\alpha\jp{x}^{-|\alpha|}|\xi|^b.
\]
If $b=1$, we write $\sigma(x,\xi)\sim |\xi|$.
Then Proposition \ref{LAP} given in Section \ref{Section2} induces
the following uniform resolvent estimate:
\medskip
\begin{thm}\label{RE}
Let $n\ge2$ and let $\theta\in\R$.
Suppose that $\sigma(x,\xi)\sim|\xi|^{\theta}$,
$\tau(x,\xi)\sim|\xi|^{1-\theta}$,
and $\sigma(x,\xi)=0$ on $\Gamma$.
Then we have the estimate
\begin{equation}\label{eq:res3}
\sup_{\lambda\in\R}
\n{
\sigma(X,D)
R(\lambda\pm i0)\tau(X,D)^*
v}_{L^2_{-1/2+l}(\R^n)}
\le
C\n{v}_{L^2_{1/2+l}(\R^n)}
\end{equation}
$0<l<\min\b{1,(n-1)/2}$.
Suppose also that $\tau(x,\xi)=0$ on $\Gamma$.
Then estimate \eqref{eq:res3} is true for
$|l|<\min\b{1,(n-1)/2}$.
\end{thm}
\medskip
As a special case of Theorem \ref{RE}, we have
\medskip
\begin{cor}\label{REcor}
Let $n\ge2$.
Suppose that $\sigma(x,\xi)\sim|\xi|^{1/2}$ and
$\sigma(x,\xi)=0$ on $\Gamma$.
Then we have
\begin{equation}\label{eq:res2}
\sup_{\lambda\in\R}
\n{
\sigma(X,D)
R(\lambda\pm i0)\sigma(X,D)^*
v}_{L^2_{-1/2}(\R^n)}
\le
C\n{v}_{L^2_{1/2}(\R^n)}.
\end{equation}
Suppose that $\sigma(x,\xi)\sim|\xi|$ and
$\sigma(x,\xi)=0$ on $\Gamma$.
Then we have
\begin{equation}\label{eq:res1}
\sup_{\lambda\in\R}
\n{
\sigma(X,D)
R(\lambda\pm i0)v}_{L^2_{-1/2+l}(\R^n)}
\le
C\n{v}_{L^2_{1/2+l}(\R^n)}
\end{equation}
for $0<l<1$ if $n\geq3$ and $0<l<1/2$ if $n=2$.
\end{cor}
\medskip
We now prove Theorem \ref{RE}.
We may consider only the case of non-negative $l$ 
since the estimate for
negative $l$ is also given by the duality argument.
That is, it suffices to show estimate \eqref{eq:res3} 
for $0<l<\min\b{1,(n-1)/2}$
assuming that $\sigma(x,\xi)=0$ on $\Gamma$, and for $l=0$
assuming that $\sigma(x,\xi)=\tau(x,\xi)=0$ on $\Gamma$.
We split estimate \eqref{eq:res3} into the following two estimates:
\begin{equation}\label{eq:res2-1}
\sup_{\lambda\neq0}
\n{
\sigma(X,D)
R(\lambda\pm i0)\tau(X,D)^*
v}_{L^2_{-1/2+l}(\R^n)}
\le
C\n{v}_{L^2_{1/2+l}(\R^n)},
\end{equation}
\begin{equation}\label{eq:res2-2}
\n{
\sigma(X,D)
R(0\pm i0)
\tau(X,D)^*
v}_{L^2_{-1/2+l}(\R^n)}
\le
C\n{v}_{L^2_{1/2+l}(\R^n)}.
\end{equation}

The proof of estimate \eqref{eq:res2-2} is reduced to
show the $L^2$-boundedness of the operator
$A(X,D,Y)=\jp{X}^{-1/2+l}\sigma(X,D)p(D)^{-2}\tau(X,D)^*\jp{Y}^{-1/2-l}$.
Since $\jp{x}^{-1/2+l}\leq C\p{|x|^{-1/2}+|x|^{-1/2+l}}$ and
$\jp{x}^{-1/2-l}\leq\min\b{|x|^{-1/2},|x|^{-1/2-l}}$, it is further
reduced to that of
$A(X,D,Y)=\abs{X}^{-1/2+l}\sigma(X,D)p(D)^{-2}\tau(X,D)^*\abs{Y}^{-1/2-l}$
with $0\leq l<(n-1)/2$, which is obtained from the following lemma
(with $b=1$ and $\delta=1/2-l$):
\medskip
\begin{lem}\label{lemma2}
Let $\delta<n/2$, $0<b<\delta+n/2$.
Suppose that
$A(x,\cdot,y)\in C^\infty\p{\R^n\setminus0}$
and 
\[
A(x,\lambda\xi,y)=\lambda^{-b}A(x,\xi,y)\qquad
(\lambda>0,\,x,y\in\R^n\setminus0).
\]
Then we have
\[
\n{A(X,D,Y)u}_{L^2(\R^n)}
\leq C
\p{
\sup_{\substack{x,y\in\R^n\setminus0\\ |\xi|=1,\,|\gamma|\leq 2n}}
\abs{
|x|^{\delta}\partial_\xi^\gamma A(x,\xi,y)|y|^{b-\delta}
}
}
 \n{u}_{L^2(\R^n)}.
\]
\end{lem}
\begin{proof}
We have
\[
A(X,D,Y)u(x)=\int K(x,x-y,y)u(y)\,dy,
\]
where
\[
K(x,z,y)=\mathcal F_\xi^{-1}[A(x,\xi,y)](z).
\]
Taking a cutoff function $\chi(\xi)\in C^\infty_0(\R^n)$ such that
$\chi(\xi)=1$ near $\xi=0$, we have
\begin{align*}
K(x,z,y)&=
(2\pi)^{-n}\int e^{i\xi\cdot z}A(x,\xi,y)\chi(\xi)\,d\xi
\\
&+
(2\pi)^{-n}|z|^{-2n}\int e^{i\xi\cdot z}
(-\triangle_\xi)^n\p{A(x,\xi,y)(1-\chi)(\xi)}\,d\xi
\end{align*}
by integration by parts.
Then we have
\[
\sup_{|z|=1}\abs{K(x,z,y)}
\leq C
\sup_{\substack{\xi\neq0\\ |\gamma|\leq 2n}}
|\xi|^{b+|\gamma|}
\abs{
\partial_\xi^\gamma A(x,\xi,y)
}
=C
\sup_{\substack{|\xi|=1\\ |\gamma|\leq 2n}}
\abs{
\partial_\xi^\gamma A(x,\xi,y)
},
\]
hence
\[
\abs{K(x,z,y)}
\leq C
\sup_{\substack{|\xi|=1\\ |\gamma|\leq 2n}}
\abs{
\partial_\xi^\gamma A(x,\xi,y)
}
\,|z|^{-(n-b)}
\]
since
$K(x,\lambda z,y)=\lambda^{b-n}K(x,z,y)$ for $\lambda>0$ and $x,y\in\R^n$.
From this, we obtain
\[
\abs{A(X,D,Y)u(x)} \leq C
\p{
\sup_{\substack{x,y\in\R^n\setminus0 \\ |\xi|=1, |\gamma|\leq 2n}}
\abs{
|x|^{\delta}\partial_\xi^\gamma A(x,\xi,y)|y|^{b-\delta}
}
}
\int\frac{|u(y)|}{|x|^{\delta}|x-y|^{n-b}|y|^{b-\delta}}\,dy,
\]
which implies the result by Lemma \ref{SW}.
\end{proof}
\medskip
We show estimate \eqref{eq:res2-1} by the scaling argument.
Noting that we have generally
\[
a(X,D,Y)f(x)=a(|\lambda|^{-1}X,|\lambda| D,|\lambda|^{-1}Y)
[f(|\lambda|^{-1}\cdot)](|\lambda| x),
\]
estimate \eqref{eq:res2-1} is reduced to showing the estimates
\begin{equation}\label{posres}
\sup_{\lambda>0}
\n{
\sigma_\lambda(X,D)
R(1\pm i0)
\tau_\lambda(X,D)^*
v
}_{L^2\p{\R^n}}
\le
C
\n{v}_{L^2\p{\R^n}},
\end{equation}
\begin{equation}\label{negres}
\sup_{\lambda>0}
\n{
\sigma_\lambda(X,D)
R(-1\pm i0)
\tau_\lambda(X,D)^*
v
}_{L^2\p{\R^n}}
\le
C
\n{v}_{L^2\p{\R^n}},
\end{equation}
where we set
\begin{align*}
&\sigma_\lambda(x,\xi)=\lambda^{-1/2}\jp{\lambda^{-1}x}^{-1/2+l}
\sigma(\lambda^{-1}x,\xi),
\\
&\tau_\lambda(x,\xi)=\lambda^{-1/2}\jp{\lambda^{-1}x}^{-1/2-l}
\tau(\lambda^{-1}x,\xi).
\end{align*}
If $l>1/2$, for $x$ such that $|\lambda^{-1}x|\geq 2$, we have
$\lambda^{-1/2}\jp{\lambda^{-1}x}^{-1/2+l}\leq C\lambda^{-l}|x|^{-1/2+l}$,
and for $x$ such that $|\lambda^{-1}x|\leq 2$, we have
$\lambda^{-1/2}\jp{\lambda^{-1}x}^{-1/2+l}
\leq C\lambda^{-1/2}
\jp{\lambda^{-1}x}^{-1/2+l'}\leq C\lambda^{-l'}|x|^{-1/2+l'}$
for $l'\leq 1/2$.
Hence we have
\begin{align*}
&\n{
\sigma_\lambda(X,D)
R(\pm1\pm i0)
\tau_\lambda(X,D)^*
v
}_{L^2\p{\R^n}}
\\
\leq&
\p{
\p{\int_{|\lambda^{-1}x|\geq 2}+\int_{|\lambda^{-1}x|\leq 2}}
\abs{
\sigma_\lambda(X,D)
R(\pm1\pm i0)
\tau_\lambda(X,D)^*
v
}^2
\,dx
}^{1/2}
\\
\leq& C
\n{
\lambda^{-l}
|X|^{-1/2+l}
\sigma(\lambda^{-1}X,D)R(\pm1\pm i0)
\tau_\lambda(X,D)^*
v
}_{L^2\p{\R^n}}
\\
+& C
\n{
\lambda^{-l'}
|X|^{-1/2+l'}
\sigma(\lambda^{-1}X,D)R(\pm1\pm i0)
\tau_\lambda(X,D)^*
v
}_{L^2\p{\R^n}}.
\end{align*}
If $0\leq l\leq1/2$, then we have
$\lambda^{-1/2}\jp{\lambda^{-1}x}^{-1/2+l}\leq \lambda^{-l}|x|^{-1/2+l}$
and
\begin{align*}
&\n{
\sigma_\lambda(X,D)
R(\pm1\pm i0)
\tau_\lambda(X,D)^*
v
}_{L^2\p{\R^n}}
\\
\leq& C
\n{
\lambda^{-l}
|X|^{-1/2+l}
\sigma(\lambda^{-1}X,D)R(\pm1\pm i0)
\tau_\lambda(X,D)^*
v
}_{L^2\p{\R^n}}.
\end{align*}
Furthermore, since we have
$\lambda^{-1/2}\jp{\lambda^{-1}x}^{-1/2-l}\leq
\lambda^{-1/2}\jp{\lambda^{-1}x}^{-1/2-l'}\leq
\lambda^{l'}\abs{x}^{-1/2-l'}$ for $l'\leq l$,
we may replace $\tau_\lambda(x,\xi)$ by
$\lambda^{l}|x|^{-1/2-l}\tau(\lambda^{-1}x,\xi)$
or
$\lambda^{l'}|x|^{-1/2-l'}\tau(\lambda^{-1}x,\xi)$
whichever we like,
in the right hand sides of these estimates.

On account of them, it suffices to show estimates \eqref{posres}
and \eqref{negres} for
$\sigma_\lambda(x,\xi)$ and $\tau_\lambda(x,\xi)$
of the forms
\begin{equation}\label{sigmatau}
\sigma_\lambda(x,\xi)=|x|^{-1/2+l}\sigma(\lambda^{-1}x,\xi),\quad
\tau_\lambda(x,\xi)=|x|^{-1/2-l}\tau(\lambda^{-1}x,\xi).
\end{equation}
for $0<l<\min\b{1,(n-1)/2}$
assuming $\sigma_\lambda(x,\xi)=0$ on $\Gamma$, and for $l=0$
assuming $\sigma_\lambda(x,\xi)=\tau_\lambda(x,\xi)=0$ on $\Gamma$.
We remark that $\sigma_\lambda(x,\xi)$ and $\tau_\lambda(x,\xi)$
defined by \eqref{sigmatau} satisfy the estimates
\begin{equation}\label{sigtauest}
\begin{aligned}
&|\partial_x^\alpha\partial_\xi^\beta
\p{\sigma_\lambda(x,\xi)}|
\leq C_{\alpha\beta}\abs{x}^{-1/2+l-\alpha}|\xi|^{\theta-\beta},
\\
&|\partial_x^\alpha\partial_\xi^\beta
\p{\tau_\lambda(x,\xi)}|
\leq C_{\alpha\beta}\abs{x}^{-1/2-l-\alpha}|\xi|^{1-\theta-\beta}
\end{aligned}
\end{equation}
with constants $C_{\alpha\beta}$ independent of $\lambda>0$.

We split estimate \eqref{posres}
into the following two parts:
\begin{equation}\label{high}
\sup_{\lambda>0}
\n{
\sigma_\lambda(X,D)
R(1\pm i0)
\p{1-\chi\circ p}(D)
\tau_\lambda(X,D)^*
v
}_{L^2\p{\R^n}}
\le
C
\n{v}_{L^2\p{\R^n}},
\end{equation}
\begin{equation}\label{low}
\sup_{\lambda>0}
\n{
\sigma_\lambda(X,D)
R(1\pm i0)
\p{\chi\circ p}(D)
\tau_\lambda(X,D)^*
v
}_{L^2\p{\R^n}}
\le
C
\n{v}_{L^2\p{\R^n}},
\end{equation}
where $\chi(t)\in \mathcal{C}^\infty_0\p{\R_+}$ is a function
which is equal to $1$ near $t=1$.
Estimates  \eqref{negres} and \eqref{high} are obtained if we write
\begin{align*}
&\sigma_\lambda(X,D)R(-1)
\tau_\lambda(X,D)^*
=\tilde\sigma_\lambda(X,D,Y)m(X,D,Y)\tilde\tau_\lambda(X,D,Y)^*,
\\
&\sigma_\lambda(X,D)R(1\pm i0)
\p{1-\chi\circ p}(D)\tau_\lambda(X,D)^*
=\tilde\sigma_\lambda(X,D,Y)\tilde m(X,D,Y)\tilde\tau_\lambda(X,D,Y)^*,
\end{align*}
where
\[
\tilde\sigma_\lambda(X,D,Y)=\sigma_\lambda(X,D)|D|^{-\theta-1/2}|Y|^{-l},
\quad
\tilde\tau_\lambda(X,D,Y)=\tau_\lambda(X,D)|D|^{\theta-3/2}|Y|^l,
\]
and
\begin{align*}
& m(X,D,Y)=|X|^{l}|D|^2R(-1)
|Y|^{-l},
\\
& \tilde m(X,D,Y)=|X|^{l}|D|^2R(1\pm i0)
\p{1-\chi\circ p}(D)|Y|^{-l}.
\end{align*}
All of them are $L^2$-bounded (uniformly in $\lambda>0$) by
Lemma \ref{lemma2} (with $b=1/2$ and $\delta=1/2\mp l$)
and estimates \eqref{sigtauest}
together with the following lemma by
Kurtz and Wheeden \cite[Theorem 3]{KW}:
\medskip
\begin{lem}\label{lemma1}
Let $-n/2<\delta<n/2$.
Then we have
\begin{align*}
&\n{|x|^\delta m(D)u}_{L^2(\R^n)}
\le C
\sum_{|\gamma|\leq n}
\sup_{\xi\in\R^n}
\abs{
|\xi|^{|\gamma|}\partial^\gamma m(\xi)
}
\n{|x|^\delta u
}_{L^2(\R^n)},
\\
&\n{\jp{x}^\delta m(D)u}_{L^2(\R^n)}
\le C
\sum_{|\gamma|\leq n}
\sup_{\xi\in\R^n}
\abs{
|\xi|^{|\gamma|}\partial^\gamma m(\xi)
}
\n{\jp{x}^\delta u
}_{L^2(\R^n)}.
\end{align*}
\end{lem}
\begin{proof}
The first estimate is due to \cite{KW}.
The second estimate with $0\leq\delta<n/2$ is obtained from it because of
the inequality $\jp{x}^\delta\leq C(1+|x|^\delta)$.
The one with $-n/2<\delta\leq0$ is just the dual of it.
\end{proof}
\medskip

\par
We prove estimate (\ref{low}).
Let $\rho(x)\in C^\infty_0\p{\R^n}$ be equal to 1 near the origin
and $\tilde{\chi}(t)\in C^\infty_0\p{\R_+}$ be equal to 1 on $\supp \chi$.
We set
\begin{align*}
&\sigma_\lambda^0(x,\xi)=\rho(x)\sigma_\lambda(x,\xi),\quad
\sigma_\lambda^1(x,\xi)=
\p{1-\rho(x)}\sigma_\lambda(x,\xi)\p{\tilde{\chi}\circ p}(\xi),
\\
&\tau_\lambda^0(x,\xi)=\rho(x)\tau_\lambda(x,\xi),\quad
\tau_\lambda^1(x,\xi)=
\p{1-\rho(x)}\tau_\lambda(x,\xi)\p{\tilde{\chi}\circ p}(\xi).
\end{align*}
By Proposition \ref{LAP} and estimates \eqref{sigtauest}, we have
\begin{equation}\label{sigma11}
\sup_{\lambda>0}
\n{
\sigma_\lambda^1(X,D)
R(1\pm i0)
\p{\chi\circ p}(D)
\tau_\lambda^1(X,D)^*
v
}_{L^2\p{\R^n}}
\le
C
\n{v}_{L^2\p{\R^n}}.
\end{equation}

Other estimates are obtained form the following lemma
which was proved by \cite[Theorem 1.2]{SuT}
(see also \cite[Theorem 3.1]{Su1}):
\medskip
\begin{lem}\label{SuTsu}
Let $1-n/2<\alpha<1/2$ and $1-n/2<\beta<1/2$.
Then we have
\[
\n{
\abs{x}^{\alpha-1}
\abs{D}^{\alpha+\beta}
R(1\pm i0)
\abs{x}^{\beta-1}
v
}_{L^2\p{\R^n}}
\le
C
\n{v}_{L^2\p{\R^n}}.
\]
\end{lem}
\medskip
In fact, from this lemma
and Lemma \ref{lemma1} with
$m(\xi)=f(|\xi|)|\xi|^{-(\alpha+\beta)}\p{\chi\circ p}(\xi)^{-1}$,
we obtain the estimate
\begin{equation}\label{simple}
\n{
\abs{x}^{\alpha-1}
R(1\pm i0)
f(|D|)
\abs{x}^{\beta-1}
v
}_{L^2\p{\R^n}}
\le
C_f
\n{v}_{L^2\p{\R^n}}
\end{equation}
for $f\in C^\infty_0(\R_+)$.
Since $\sigma_\lambda^0(X,D)|D|^{-(\theta+\varepsilon)}|Y|^{1/2+\varepsilon}$
and $\tau_\lambda^0(X,D)|D|^{-(1-\theta+\varepsilon)}|Y|^{1/2+\varepsilon}$ 
are $L^2$-bounded uniformly in $\lambda>0$ by Lemma \ref{lemma2}
(with $b=\varepsilon$ and $\delta=1/2+2\epsilon$),
where $0<\varepsilon<(n-1)/4$ and $l/2\leq\varepsilon$,
the estimate
\begin{equation}\label{sigma00}
\sup_{\lambda>0}
\n{
\sigma_\lambda^0(X,D)
R(1\pm i0)
\p{\chi\circ p}(D)
\tau_\lambda^0(X,D)^*
v
}_{L^2\p{\R^n}}
\le
C
\n{v}_{L^2\p{\R^n}}
\end{equation}
is reduced to the estimate
\[
\n{
\abs{x}^{-(1/2+\varepsilon)}
R(1\pm i0)
\p{\chi\circ p}(D)
|D|^{1+2\varepsilon}
\abs{x}^{-(1/2+\varepsilon)}
v
}_{L^2\p{\R^n}}
\le
C
\n{v}_{L^2\p{\R^n}}
\]
which is a special case of estimate \eqref{simple}
with $\alpha=\beta=1/2-\varepsilon$.
Similarly, we have that
$\tau_\lambda^1(X,D)|D|^{-(1-\theta+\varepsilon)}|Y|^{1/2+\varepsilon}$
with $\epsilon=l/2$ is $L^2$-bounded uniformly in
$\lambda>0$, and the estimate
\begin{equation}\label{sigma01}
\sup_{\lambda>0}
\n{
\sigma_\lambda^0(X,D)
R(1\pm i0)
\p{\chi\circ p}(D)
\tau_\lambda^1(X,D)^*
v
}_{L^2\p{\R^n}}
\le
C
\n{v}_{L^2\p{\R^n}}
\end{equation}
with $l>0$ is also reduced to estimate \eqref{simple}.

Hence all the rest to be shown is the estimate
\begin{equation}\label{sigma10}
\sup_{\lambda>0}
\n{
\sigma_\lambda^1(X,D)
R(1\pm i0)
\p{\chi\circ p}(D)
\tau_\lambda^0(X,D)^*
v
}_{L^2\p{\R^n}}
\leq
C\n{v}_{L^2\p{\R^n}}.
\end{equation}
We remark that the estimate \eqref{sigma01} with $l=0$
is just the dual of estimate \eqref{sigma10} with $l=0$.
By Lemmas \ref{basiclem} and Lemma \ref{Com},
estimate \eqref{sigma10}
is reduced to the estimate
\[
\sup_{\lambda>0}
\n{
\jp{x}^{-3/2+l}
R(1\pm i0)
\p{\chi\circ p}(D)
\p{\Omega_{ij}}^k
\tau_\lambda^0(X,D)^*
v
}_{L^2\p{\R^n}}
\leq
C\n{v}_{L^2\p{\R^n}},
\]
where $k=0,1$.
Since the symbol of $\Omega_{ij}$ is linear in $x$ and
positively homogeneous of order $1$ in $\xi$ by \eqref{Omega},
$\tau_\lambda^0(X,D)\p{\Omega_{ij}^*}^k$ is a finite sum of the
operators of the form
$\rho(X)|X|^{-1/2-l+\mu}\tilde\tau_\lambda(X,D)|D|^\mu$,
where $\mu=0,1$ and
$\tilde\tau_\lambda(x,\xi)$ is
homogeneous of orders $1-\theta$ in $\xi$.
Furthermore
$\rho(X)|X|^{-1/2-l+\mu}\tilde\tau_\lambda(X,D)|D|^\mu
|D|^{-(1-\theta+\varepsilon+\mu)}|Y|^{1/2+\varepsilon}$
are $L^2$-bounded uniformly in $\lambda>0$ by Lemma \ref{lemma2}
(with $b=\varepsilon$ and $\delta=1/2+2\epsilon$),
where $0<\varepsilon<(n-1)/4$ and $l/2\leq\varepsilon$.
Noticing the trivial inequality
$\jp{x}^{-3/2+l}\leq\jp{x}^{l/2-1}\leq|x|^{l/2-1}$,
the estimate is further reduced to
\[
\n{
\abs{x}^{l/2-1}
R(1\pm i0)
\p{\chi\circ p}(D)
|D|^{1-\theta+\varepsilon+\mu}
\abs{x}^{-(1/2+\varepsilon)}
v
}_{L^2\p{\R^n}}
\le
C
\n{v}_{L^2\p{\R^n}}
\]
which are implied again by \eqref{simple}
with $\alpha=l/2$ and $\beta=1/2-\varepsilon$.

Summing up estimates \eqref{sigma11}, \eqref{sigma00},
\eqref{sigma01}, and \eqref{sigma10}, we have estimate \eqref{low},
and thus 
the proof of Theorem \ref{RE} is complete.

\medskip
\par
\section{Herbst-Skibsted's resolvent estimate}\label{Section4}
In this section, we will explain the relation between
Theorem \ref{RE}
and Herbst-Skibsted's resolvent estimate in \cite{HS}.
\par
Let $S(x,\lambda)$ be the solution of eikonal equation,
\[
p(\nabla S(x,\lambda))^2+V(x)=\lambda\quad(\lambda>0)
\]
for $L=p(D)^2+V(x)$. 
Here we always assume $V=0$ but keep it remaining
in the notation
because the case $V\neq0$ is admitted in \cite{HS}
assuming the potential $V$ to be smooth and to
have the property
\[
|\partial^\alpha V(x)|\leq C_\alpha\jp{x}^{-\varepsilon_0-|\alpha|}
\]
where $0<\varepsilon_0<1$.
In this case, we have $S(x,\lambda)=\sqrt{\lambda}p^*(x)$,
where $p^*(x)$ is the dual function of $p(\xi)$ defined 
by satisfying
the relations $p(\nabla p^*(x))=1$ and
$\nabla p^*(\nabla p(\xi))=\xi/p(\xi)$ 
(see \cite[Theorem 3.1]{RS2}).
Noting that $D=-i\nabla$ is the momentum operator, we set
\[
\gamma(\lambda)=D\mp\nabla S(x,\lambda).
\]
The quantisation of $\gamma(\lambda)$ is given by
\[
\bar\gamma=D\mp\nabla S(x,p(D)^2).
\]
We remark that the symbol 
$\bar\gamma(x,\xi)=\xi\mp p(\xi)\nabla p^*(x)$ 
of the operator
$\bar\gamma$ satisfies the half structure condition
\begin{equation}\label{EQ:half-str}
\text{$\bar\gamma(x,\xi)=0$ on $\Gamma^\pm$},
\end{equation} 
where
\[
\Gamma^\pm
=\b{\p{\lambda\nabla p(\xi),\xi}:\,
 x\in\R^n\setminus 0,\,\pm\lambda>0}.
\]
In the case $L=-\Delta$, for example, we have
\[
\gamma(\lambda)=D \mp\sqrt{\lambda}\frac x{|x|},
\quad
\bar\gamma=D \mp\frac x{|x|}|D|,
\]
and the following results are already known, as
an adapted version of those in \cite{HS}:
\medskip
\begin{thm}[{\cite[Theorems 4.4, 5.1]{HS}}]\label{HSres}
Let $0\leq s\leq1$, let $\delta>1/2$, and let $\chi\in C^\infty_0(\R_+)$.
Assume $L=-\Delta+V$.
Then, we have a {\it quantum} result
\[
\sup_{\lambda\in\R}
\n{\bar\gamma\, R(\lambda\pm i0)\chi(|D|)v}_{L^2_{-\delta+s}}
\leq C\n{v}_{L^2_{\delta+s}}
\]
and a {\it classical} result
$
\gamma(\lambda)R(\lambda\pm i0)\in\mathcal{L}(L^2_{\delta+s},L^2_{-\delta+s}).
$
\end{thm}
\medskip
The quantum result in Theorem \ref{HSres} with $s=0$ 
is a usual resolvent
estimate, but we can extend it to the case 
$0<s\leq1$ by virtue of the half
structure \eqref{EQ:half-str} of the operator $\bar\gamma$. 
We remark that the classical result in Theorem \ref{HSres}
was first proved by Isozaki \cite{I}, 
and it can be also derived from
the quantum result in Theorem \ref{HSres}.
Furthermore, it implies Sommerfeld's radiation condition
\[
 (\partial_r\mp i\sqrt{\lambda})u\in L^2_{-\delta+s}
\]
for the outgoing and incoming solutions $u=R(\lambda\pm i0)v$
to the Helmholtz equation
\[
(-\triangle-\lambda)u
=v,\quad(\lambda>0,\,v\in L^2_{\delta+s})
\]
since
$i(x/|x|)\cdot\gamma(\lambda)=\partial_r\mp i\sqrt{\lambda}$.

Theorem \ref{HSres}
means that each specified operator $\sigma^\pm(X,D)=\bar\gamma$
with half structure
$\sigma^\pm(x,\xi)=0$ on $\Gamma^\pm$ implies
the estimates for $R(\lambda+i0)$ and $R(\lambda-i0)$, respectively.
On the other hand, our Theorem \ref{RE} means that
any operator $\sigma(X,D)$ with full structure
$\sigma(x,\xi)=0$ on $\Gamma$ implies estimates for both.

Furthermore, Theorem \ref{RE} corresponds to
the critical case of Theorem \ref{HSres}.
In fact, the quantum estimate in Theorem \ref{HSres}
can be interpreted as
\[
\sup_{\lambda\in\R}
\n{\sigma^\pm(X,D)R(\lambda\pm i0)\chi(|D|)v}_{L^2_{-1/2+l}}
\leq C\n{v}_{L^2_{1/2+l+2\varepsilon}}
\qquad(-\varepsilon\leq l<1,\,\varepsilon>0)
\]
for $\sigma^\pm(X,D)=\bar\gamma$,
while estimate \eqref{eq:res1} in Corollary \ref{REcor} 
(together with Lemma \ref{lemma1})
implies a
better estimate
\[
\sup_{\lambda\in\R}
\n{\sigma(X,D)R(\lambda\pm i0)\chi(|D|)v}_{L^2_{-1/2+l}}
\leq C\n{v}_{L^2_{1/2+l}}
\qquad(0<l<1)
\]
for $\sigma(X,D)$ satisfying the full structure condition
(in the case $n\geq3$). 

As another advantage of Theorem \ref{RE}, we
can treat the general operator $L=p(D)^2+V$ instead of $-\Delta+V$
although we have to assume $V=0$.

\medskip
\par
\section{Smoothing estimates for Schr\"odinger equations}\label{Section5}
It is known that uniform resolvent estimates
straightforwardly
induce smoothing estimates for
Schor\"odinger evolution operators.
For example, estimate \eqref{eq:res2} in Corollary \ref{REcor}
says that the operator $A=\jp{x}^{-1/2}\sigma(X,D)$
is $L$-supersmooth on the separable Hilbert space $H=L^2$, that is,
$A$ satisfies 
\[
\sup_{\lambda\in\R}
\n{
A
R(\lambda\pm i0)
A^*
v}_{H}
\le
C\n{v}_{H}.
\]
Then by the work of Kato \cite{K},
we have the estimate
\[
\int\n{Au}_{H}^2\,dt
\leq C\n{\varphi}_{H}^2
\]
for the solution $u(t,x)=e^{-itL}\varphi(x)$ to
\begin{equation}\label{eq:hom}
\left\{
\begin{aligned}
\p{i\partial_t-L}\,u(t,x)&=0\\
u(0,x)&=\varphi(x).
\end{aligned}
\right.
\end{equation}
Hence we equivalently have
\medskip
\begin{thm}\label{thm:thm611}
Let $n\geq2$.
Suppose that $\sigma(x,\xi)\sim|\xi|^{1/2}$ and
$\sigma(x,\xi)=0$ on $\Gamma$.
Then the solution $u$ to equation \eqref{eq:hom}
satisfies the estimate
\[
\n{\jp{x}^{-1/2}\sigma(X,D) u}_{L^2\p{\R_t\times\R^n_x}}
\leq C
\n{\varphi}_{L^2(\R^n_x)}.
\]
\end{thm}
\medskip
We remark that Theorem \ref{thm:thm611} is a refinement of the
following well known smoothing estimate
\begin{equation}\label{est:sm1}
\n{\jp{x}^{-1/2-l}|D|^{1/2} u}_{L^2\p{\R_t\times\R^n_x}}
\leq C
\n{\varphi}_{L^2(\R^n_x)}\qquad(l>0)
\end{equation}
for the solution $u$ to equation \eqref{eq:hom}
(see, for example, \cite{BK} and \cite{Ch2}).
Theorem \ref{thm:thm611} covers its critical case $l=0$
under the structure condition $\sigma(x,\xi)=0$ on $\Gamma$.
In our previous work \cite{RS2},
we also proved the estimate
\begin{equation}\label{est:ky1}
\n{\abs{x}^{-1/2}\sigma(X,D) u}_{L^2\p{\R_t\times\R^n_x}}
\leq C
\n{\varphi}_{L^2(\R^n_x)}
\end{equation}
when $\sigma(x,\xi)$ satisfies the same
structure condition, and is positively homogeneous
of order $0$ in $x$ and $1/2$ in $\xi$.
Estimate \eqref{est:ky1} is a refinement of the estimate
\[
\n{\abs{x}^{\alpha-1}|D|^{\alpha} u}_{L^2\p{\R_t\times\R^n_x}}
\leq C
\n{\varphi}_{L^2(\R^n_x)}\qquad(1-n/2<\alpha<1/2)
\]
by \cite{KY} and \cite{Su1}.
\par
On the other hand, as it
is discussed in \cite{Su1}, \cite{Su2}, \cite{SuT},
we can construct the solution $u(t,x)$
to the inhomogeneous equation
\begin{equation}\label{eq:inhom}
\left\{
\begin{aligned}
\p{i\partial_t-L}\,u(t,x)&=f(t,x)\\
u(0,x)&=0
\end{aligned}
\right.
\end{equation}
as
\[
u(t,x)
=\frac1i
\mathcal F^{-1}_\tau
R(-\tau+i0)
\mathcal F_t
f^+(t,x)
+
\frac1i
\mathcal F^{-1}_\tau
R(-\tau-i0)
\mathcal F_t
f^-(t,x).
\]
Here $f^\pm$ denotes the function $f$ multiplied by the 
Heaviside function $Y(\pm t)$, that is, the
characteristic
function of the set $\{t: \pm t\geq0\}$.
Hence estimate \eqref{eq:res1} in Corollary \ref{REcor} yields
the following result for \eqref{eq:inhom},
which is a refinement of the estimate
\begin{equation}\label{est:sm2}
\n{\jp{x}^{-1/2-l}|D| u}_{L^2\p{\R_t\times\R^n_x}}
\leq C
\n{\jp{x}^{1/2+l}f}_{L^2\p{\R_t\times\R^n_x}}\qquad(l>0)
\end{equation}
(see, for example, \cite{Ch2}):
\medskip
\begin{thm}\label{thm:thm62}
Let $n\geq2$.
Suppose that $\sigma(x,\xi)\sim|\xi|$ and
$\sigma(x,\xi)=0$ on $\Gamma$.
Then the solution $u$ to equation \eqref{eq:inhom}
satisfies the estimate
\[
\n{\jp{x}^{-1/2+l}\sigma(X,D) u}_{L^2\p{\R_t\times\R^n_x}}
\leq C
\n{\jp{x}^{1/2+l}f}_{L^2\p{\R_t\times\R^n_x}}
\]
for $0<l<1$ if $n\geq3$ and $0<l<1/2$ if $n=2$.
\end{thm}
\medskip
If we drop the structure assumption $\sigma(x,\xi)=0$ on $\Gamma$
from Theorems \ref{thm:thm611}-\ref{thm:thm62},
we cannot expect the same estimates there but can still
show the following
weaker ones:
\medskip
\begin{prop}\label{prop:nonst}
Let $n\geq2$ and let $l>0$.
Suppose that $\sigma(x,\xi)\sim|\xi|^{1/2}$.
Then the solution $u$ to equation \eqref{eq:hom}
satisfies the estimate
\[
\n{\jp{x}^{-1/2-l}\sigma(X,D) u}_{L^2\p{\R_t\times\R^n_x}}
\leq C
\n{\varphi}_{L^2(\R^n_x)}.
\]
Suppose that $\sigma(x,\xi)\sim|\xi|$.
Then the solution $u$ to equation \eqref{eq:inhom}
satisfies the estimate
\[
\n{\jp{x}^{-1/2-l}\sigma(X,D) u}_{L^2\p{\R_t\times\R^n_x}}
\leq C
\n{\jp{x}^{1/2+l}f}_{L^2\p{\R_t\times\R^n_x}}.
\]
\end{prop}
\begin{proof}
We may assume $0<l<(n-1)/2$ 
since the estimate with $l\geq(n-1)/2$ is
a weaker result.
Let $0<l'<l<(n-1)/2$ and let us
factorise $\jp{x}^{-1/2-l}\sigma(X,D)$ as
\begin{align*}
\jp{x}^{-1/2-l}\sigma(X,D)
&=\jp{x}^{-1/2-l}\sigma(X,D)|D|^{-1/2}\jp{x}^{{1/2+l'}}
\cdot
\jp{x}^{{-1/2-l'}}|D|^{1/2}
\\
&=\jp{x}^{-1/2-l}\sigma(X,D)|D|^{-1}\jp{x}^{{1/2+l'}}
\cdot
\jp{x}^{{-1/2-l'}}|D|.
\end{align*}
On account of estimates \eqref{est:sm1} and \eqref{est:sm2},
it suffices to show the $L^2$-boundedness of the operator
$\jp{x}^{-1/2-l}\sigma(X,D)\jp{x}^{{1/2+l'}}$
assuming $\sigma(x,\xi)\sim|\xi|^0$.
Let $\chi(\xi)\in \mathcal{C}^\infty_0\p{\R^n}$ be a function
which is equal to $1$ near the origin.
Then by the symbolic calculus and the $L^2$-boundedness
of pseudo-differential operators of class $S^0$
(see also \cite[Theorem 1.1]{RS1}),
the operator $\jp{x}^{-1/2-l}\sigma(X,D)(1-\chi)(D)\jp{x}^{{1/2+l'}}$
is $L^2$-bounded.
On the other hand, the $L^2$-boundedness of the operator
$\jp{x}^{-1/2-l}\sigma(X,D)\chi(D)\jp{x}^{{1/2+l'}}$ is reduced to
that of $|x|^{-1/2-l}\sigma(X,D)\chi(D)|x|^{{1/2+l'}}$ 
and $|x|^{-1/2-l}\sigma(X,D)\chi(D)$
since $\jp{x}^{-(1/2+l)}\leq |x|^{-(1/2+l)}$
and $\jp{x}^{1/2+l'}\leq C(1+|x|^{1/2+l'})$. 
Due to Lemma \ref{lemma1}, they are further reduced to that of
$|x|^{-1/2-l}\sigma(X,D)|D|^{-(l-l')}|x|^{{1/2+l'}}$ 
and $|x|^{-1/2-l}\sigma(X,D)|D|^{-1/2-l}$, which are obtained from
Lemma \ref{lemma2} with $b=l-l',\delta=1/2+l$ and $b=1/2+l,\delta=1/2+l$,
respectively. 
\end{proof}
\medskip
We have also a result similar to
Theorem \ref{thm:thm62} for the solution to homogeneous equation
\eqref{eq:hom}: 
\medskip
\begin{thm}\label{thm:thm61}
Let $n\geq3$.
Suppose that $\sigma(x,\xi)\sim|\xi|$ and
$\sigma(x,\xi)=0$ on $\Gamma$.
Then the solution $u$ to equation \eqref{eq:hom}
satisfies the estimate
\[
\n{\jp{x}^{-1/2+l}\sigma(X,D) u}_{L^2\p{\R_t\times\R^n_x}}
\leq C
\n{\jp{x}^{\alpha l}\jp{D}^{1/2}\varphi}_{L^2(\R^n_x)}
\]
for $0<l<1$ and $\alpha>3/2$.
\end{thm}
\begin{proof}
We decompose the solution $u=e^{-itL}\varphi$ into
the following two parts:
\[
u_{\textit{low}}=e^{-itL}\chi(L)\varphi,\qquad
u_{\textit{high}}=e^{-itL}(1-\chi(L))\varphi,
\]
where $\chi\in C^\infty_0(\R)$ is a function such that
$\chi(\xi)=1$ for $|\xi|<1$, and $\chi(L)=(\chi\circ p^2)(D)$.
Let $\tilde\chi\in C^\infty_0(\R)$ be another function such that
$\tilde\chi(\xi)=1$ on $\supp \chi$.
Then we have $u_{\textit{low}}=e^{-itL}\chi(L)\tilde\varphi$,
where $\tilde\varphi=(\tilde\chi\circ p^2)(D)\varphi$.
Let
\[
L=\int\lambda dE_L(\lambda)=\int\lambda A_L(\lambda)\,d\lambda
\]
be the spectral decomposition of the self-adjoint operator $L$ with
spectral measure $E_L(\lambda)$,
and
\[
A_L(\lambda)=\frac{dE_L(\lambda)}{d\lambda}
=
\frac1{2\pi i}
\p{R(\lambda+i0)-R(\lambda-i0)}
\]
be the corresponding spectral density.
Then we can write
\[
u_{\textit{low}}=
e^{-itL}\chi(L)\tilde\varphi
=
\int e^{-it\lambda}\chi(\lambda)
A_L(\lambda)\tilde\varphi
\,d\lambda,
\]
and by Plancherel's theorem (see also 
\cite[Section XIII.7, Lemma 1]{ReS}),
estimate \eqref{eq:res1} in Corollary \ref{REcor},
and Lemma \ref{lemma1}, we have 
\begin{align*}
&\n{\jp{x}^{-1/2+l}\sigma(X,D) 
u_{\textit{low}}}_{L^2\p{\R_t\times\R^n_x}}^2
\\
=&
2\pi\int|\chi(\lambda)|^2 
 \n{\jp{x}^{-1/2+l}\sigma(X,D)A_L(\lambda)\tilde\varphi}^2
_{L^2(\R^n_x)}\,d\lambda
\\
\leq&
C
\sup_{\lambda\in\R} \n{\sigma(X,D)
\p{R(\lambda+i0)-R(\lambda-i0)}
\tilde\varphi}_{L^2_{-1/2+l}(\R^n_x)}^2
\\
\leq&
C\n{\tilde\varphi}_{L^2_{1/2+l}(\R^n_x)}^2
\leq
C\n{\jp{x}^{1/2+l}\jp{D}^{1/2}\varphi}_{L^2(\R^n_x)}^2.
\end{align*}
On the other hand, by Lemma \ref{basiclem}, Proposition \ref{prop:nonst},
Lemma \ref{SW}, and Lemma \ref{lemma1}, we have
\begin{align*}
&\n{\jp{x}^{-1/2+l}\sigma(X,D) u_{\textit{high}}}_{L^2\p{\R_t\times\R^n_x}}
=
\n{\jp{x}^{-1/2+l}\sigma(X,D)
(1-\chi(L))e^{-itL}\varphi}_{L^2\p{\R_t\times\R^n_x}}
\\
\leq&
C\p{
\sum_{i<j}
\n{\jp{x}^{-1/2-(1-l)}
e^{-itL}\Omega_{ij}(X,D) \varphi}_{L^2\p{\R_t\times\R^n_x}}
+
\n{\jp{x}^{-1/2-(1-l)} e^{-itL}\varphi}_{L^2\p{\R_t\times\R^n_x}}
}
\\
\leq&
C\p{
\sum_{i<j}
\n{\Omega_{ij}(X,D)p(D)^{-1/2} \varphi}_{L^2(\R^n_x)}
+
\n{|D|^{-1/2} \varphi}_{L^2(\R^n_x)}
}
\\
\leq&
C\n{\jp{x}\jp{D}^{1/2}\varphi}_{L^2(\R^n_x)}
\end{align*}
since $\Omega_{ij}(X,D)$ commutes with $e^{-itL}$ and
$p(D)^{-1/2}$ by Lemma \ref{Com}.
Hence, for $1/2\leq l<1$, we have
\[
\n{\jp{x}^{-1/2+l}\sigma(X,D) u}_{L^2\p{\R_t\times\R^n_x}}
\leq
C\n{\jp{x}^{1/2+l}\jp{D}^{1/2}\varphi}_{L^2(\R^n_x)}.
\]
The conclusion of the theorem is obtained if we interpolate this estimate
and the estimate
\[
\n{\jp{x}^{-1/2}\sigma(X,D) u}_{L^2\p{\R_t\times\R^n_x}}
\leq
C\n{\jp{D}^{1/2}\varphi}_{L^2(\R^n_x)},
\]
which is a direct consequence of Theorem \ref{thm:thm611} and
Lemma \ref{lemma1},
\end{proof}
\medskip
For $s,\tilde s\in\R$, let $H^{s,\tilde{s}}(\R_t\times\R^n_x)$ 
be the set of tempered
distributions $g$ on $\R_t\times\R^n_x$ such that the norm
\[
\n{g}_{H^{s,\tilde{s}}(\R_t\times\R^n_x)}
=\n{\jp{D_t}^s\jp{D_x}^{\tilde s}g}
_{L^2\p{\R_t\times\R^n_x}}
\]
is finite, where $\jp{D_t}^s=
\mathcal F_\tau^{-1}\jp{\tau}^s\mathcal F_t$ and
$\jp{D_x}^{\tilde s}=
\mathcal F_\xi^{-1}\jp{\xi}^{\tilde s}\mathcal F_x$.
Combining Theorems \ref{thm:thm62} and \ref{thm:thm61},
we have the following result:
\medskip
\begin{cor}\label{thm:thm6}
Let $n\geq3$, let $0\leq s\leq1$, and let $\tilde{s}\geq0$.
Suppose that $\sigma(x,\xi)\sim|\xi|$ and
$\sigma(x,\xi)=0$ on $\Gamma$.
Then the solution $u$ to equation
\begin{equation}\label{eq:gen}
\left\{
\begin{aligned}
\p{i\partial_t-L}\,u(t,x)&=f(t,x),\\
u(0,x)&=\varphi(x),
\end{aligned}
\right.
\end{equation}
satisfies the estimate
\begin{multline*}
\n{\jp{x}^{-1/2+l}\sigma(X,D) u}_{H^{s,\tilde{s}}(\R_t\times\R^n_x)}
\\
\leq C
\p{
\n{\jp{x}^{\alpha l}\jp{D}^{2s+\tilde{s}+1/2}\varphi}_{L^2(\R^n_x)}
+\n{\jp{x}^{1/2+l}f}_{H^{s,\tilde{s}}(\R_t\times\R^n_x)}
}
\end{multline*}
for $0<l<1$ and $\alpha>3/2$.
\end{cor}
\begin{proof}
Differentiating equation \eqref{eq:gen},
we have
\[
\left\{
\begin{aligned}
\p{i\partial_t-L}\,D_x^ku(t,x)&=D_x^kf(t,x),\\
D_x^ku(0,x)&=D^k\varphi(x),
\end{aligned}
\right.
\]
and
\[
\left\{
\begin{aligned}
\p{i\partial_t-L}\,D_tD_x^ku(t,x)&=D_tD_x^kf(t,x),\\
D_tD_x^ku(0,x)&=-LD^k\varphi(x)-D_x^kf(0,x),
\end{aligned}
\right.
\]
for non-negative integers $k$.
Then by Theorems \ref{thm:thm62} and \ref{thm:thm61},
we have the estimate
\begin{multline*}
\n{\jp{x}^{-1/2+l}\sigma(X,D) D_t^jD_x^ku}
_{L^2\p{\R_t\times\R^n_x}}
\\
\leq C
\p{
\n{\jp{x}^{\alpha l}\jp{D}^{2j+k+1/2}\varphi}_{L^2(\R^n_x)}
+
\n{\jp{x}^{1/2+l}\jp{D_t}^{j}\jp{D_x}^kf(t,x)}_{L^2(\R_t\times\R^n_x)}
}
\\
\leq C
\p{
\n{\jp{x}^{\alpha l}\jp{D}^{2j+k+1/2}\varphi}_{L^2(\R^n_x)}
+
\n{\jp{x}^{1/2+l}f(t,x)}_{H^{j,k}(\R_t\times\R^n_x)}
},
\end{multline*}
for $j=0,1$.
Here we have used Lemma \ref{lemma1}, Sobolev's
embedding $H^1(\R_t)\hookrightarrow L^\infty(\R_t)$, and
the $L^2$-boundedness of the operator
$\jp{x}^{1/2+l}\jp{D}^k\jp{x}^{-(1/2+l)}\jp{D}^{-k}$
which can be justified by the symbolic calculus and 
the $L^2$-boundedness of
pseudo-differential operators of class $S^0$.
On the other hand, the commutator $[D_x^k,\,\jp{x}^{-1/2+l}\sigma(X,D)]$
is again a pseudo-differential operator with a symbol of the form
$\jp{x}^{-1/2-(1-l)}\tau(x,\xi)$ where $\tau(x,\xi)\sim|\xi|^{k'}$
($0\leq k'\leq k$).
Hence, if we use Proposition \ref{prop:nonst} instead,
we have similarly the estimate
\begin{align*}
&\n{[D_x^k,\,\jp{x}^{-1/2+l}\sigma(X,D)]D_t^ju}
_{L^2\p{\R_t\times\R^n_x}}
\\
\leq &C
\p{
\n{\jp{D}^{2j+k-1/2}\varphi}_{L^2(\R^n_x)}
+
\n{\jp{x}^{1/2+l}f(t,x)}_{H^{j,k}(\R_t\times\R^n_x)}
}
\\
\leq &C
\p{
\n{\jp{x}^{l}\jp{D}^{2j+k+1/2}\varphi}_{L^2(\R^n_x)}
+
\n{\jp{x}^{1/2+l}f(t,x)}_{H^{j,k}(\R_t\times\R^n_x)}
}
\end{align*}
since $1-l, l>0$.
Thus we have the desired estimate if $s,\tilde s$ are integers.
By interpolation, we have the conclusion.
\end{proof}

\medskip
\par
\section{Derivative Nonlinear Schr\"odinger equations with structure}
\label{Section6}
Estimates obtained in the previous section can be used to show
a time global existence result for 
derivative nonlinear Schr\"odinger equations.
We consider the power
of the {\it derivative} $\sigma(X,D)u$ of $u(t,x)$
in the space variable $x$ as nonlinear term.
\medskip
\begin{thm}\label{thm:thm7}
Let $n\geq3$, let $N\geq4$, $N\in\N$, let $\kappa>3(N+1)/4(N-1)$,
and let $\tau>(n+3)/2$.
Suppose that $\sigma(x,\xi)\sim|\xi|$ and
$\sigma(x,\xi)=0$ on $\Gamma$.
Assume that
$\langle x\rangle^\kappa \langle D\rangle^\tau\varphi\in L^2$ and
that its $L^2$-norm is sufficiently small.
Then the equation 
\begin{equation}\label{eq:eq4}
\left\{
\begin{aligned}
\p{i\partial_t-L}\,u(t,x)=&\p{\sigma(X,D)u}^N\\
u(0,x)=\varphi(x),\quad &t\in\R,\, x\in\R^n
\end{aligned}
\right.
\end{equation}
has a time global solution $u\in C^0(\R_t\times\R^n_x)$.
\end{thm}
\medskip
The key point of the proof
is that the space $H^{s,\tilde{s}}=H^{s,\tilde{s}}(\R_t\times\R^n_x)$ is
an algebra
if $s>1/2$ and $\tilde{s}>n/2$.
Then we have
\begin{equation}\label{alg1}
\begin{aligned}
\n{\jp{x}^{1/2+l}F^N}_{H^{s,\tilde{s}}(\R_t\times\R^n_x)}
\leq&
C\n{\jp{x}^{1/(2N)+l/N}
F}_{H^{s,\tilde{s}}(\R_t\times\R^n_x)}^N
\\
\leq&
C\n{\jp{x}^{-1/2+l}F}_{H^{s,\tilde{s}}(\R_t\times\R^n_x)}^N
\end{aligned}
\end{equation}
if $l=(N+1)/(2N-2)$.
Note that $1/2<l<1$ if $N\geq4$.
Using the formula
$F^N-G^N=(F-G)(F^{N-1}+F^{N-2}G+\cdots+G^{N-1})$,
we have similarly
\begin{multline}\label{alg2}
\n{\jp{x}^{1/2+l}(F^N-G^N)}
_{H^{s,\tilde{s}}(\R_t\times\R^n_x)}
\\
\leq
C
\p{
\sum_{j=0}^{N-1}
\n{\jp{x}^{-1/2+l}F}^{N-1-j}_{H^{s,\tilde{s}}}
\cdot
\n{\jp{x}^{-1/2+l}G}^j_{H^{s,\tilde{s}}}
}
\\
\times
\n{\jp{x}^{-1/2+l}(F-G)}_{H^{s,\tilde{s}}(\R_t\times\R^n_x)}.
\end{multline}

To prove Theorem \ref{thm:thm7},
we consider the mapping from $u_0(t,x)$ to the solution $u(t,x)$ for
\begin{equation}\label{defmap}
\left\{
\begin{aligned}
\p{i\partial_t-L}\,u(t,x)=&\p{\sigma(X,D)u_0}^N\\
u(0,x)=\varphi(x),\quad &t\in\R,\, x\in\R^n.
\end{aligned}
\right.
\end{equation}
We take $1/2<s\leq1$ and $\tilde s>n/2$
such that $\tau=2s+\tilde s +1/2$,
and let us use Corollary \ref{thm:thm6} with $f=\p{\sigma(X,D)u_0}^N$
and $l=(N+1)/(2N-2)$.
On account of \eqref{alg1}, we have an estimate
\begin{multline}\label{apri11}
\n{\jp{x}^{-1/2+l}\sigma(X,D) u}_{H^{s,\tilde{s}}(\R_t\times\R^n_x)}
\\
\leq C
\p{
\n{\jp{x}^\kappa\jp{D}^{\tau}\varphi}_{L^2(\R^n_x)}
+\n{\jp{x}^{-1/2+l}\sigma(X,D)u_0}_{H^{s,\tilde{s}}(\R_t\times\R^n_x)}^N
}
\end{multline}
for $\kappa>(3/2)l$.
On the other hand, let $\tilde u$ be the solution of the equation
\[
\left\{
\begin{aligned}
\p{i\partial_t-L}\,\tilde u(t,x)=&\p{\sigma(X,D)\tilde u_0}^N\\
u(0,x)=\varphi(x),\quad &t\in\R,\, x\in\R^n.
\end{aligned}
\right.
\]
Then by Corollary \ref{thm:thm6}
with $\varphi=0$,
we have
\[
\n{\jp{x}^{-1/2+l}\sigma(X,D) 
(u-\tilde u)}_{H^{s,\tilde{s}}(\R_t\times\R^n_x)}
\leq C\n{\jp{x}^{1/2+l}
(f-\tilde f)}_{H^{s,\tilde{s}}(\R_t\times\R^n_x)},
\]
where $f=\p{\sigma(X,D)u_0}^N$
and $\tilde f=\p{\sigma(X,D)\tilde u_0}^N$.
Then, on account of \eqref{alg2},
we have
\begin{multline}\label{apri22}
\n{\jp{x}^{-1/2+l}\sigma(X,D) (u-\tilde u)}
_{H^{s,\tilde{s}}(\R_t\times\R^n_x)}
\\
\leq
C
\p{
\sum_{j=0}^{N-1}
\n{\jp{x}^{-1/2+l}\sigma(X,D)u_0}^{N-1-j}_{H^{s,\tilde{s}}}
\cdot
\n{\jp{x}^{-1/2+l}\sigma(X,D)\tilde u_0}^j_{H^{s,\tilde{s}}}
}
\\
\times
\n{\jp{x}^{-1/2+l}\sigma(X,D)(u_0-\tilde u_0)}
_{H^{s,\tilde{s}}(\R_t\times\R^n_x)}.
\end{multline}
\par
Estimates \eqref{apri11} and \eqref{apri22} show that, 
if the first term of the right hand
side of \eqref{apri11} is sufficiently small,
the mapping $u_0$ to the solution $u$ for \eqref{defmap}
is a contraction on the space $X$
which collects all functions $u(t,x)$ with sufficiently small
$\n{u}_X=\n{\jp{x}^{-1/2+l}\sigma(X,D)u}_{H^{s,\tilde{s}}
(\R_t\times\R^n_x)}$.
(Note that $\n{\cdot}_X$ is not a norm
because $\n{u}_X=0$ does not always imply $u=0$).
Let us denote this mapping by $\Phi:X\to X$
noticing that the fixed point of it is a desired time global solution.
The {\it contraction} here means that we have
\[
\n{\Phi(u)-\Phi(\tilde u)}_{X}\leq \varepsilon\n{u-\tilde u}_X
\qquad (v,\tilde u\in X)
\]
with some $0<\varepsilon<1$.
Consider the sequence of functions $\b{u_n}_{n\in\N}$ in $X$ defined by
$u_n=\Phi(u_{n-1})$, $u_0=0$.
Then we have $\n{u_n-u_m}_X\to0$ if $m,n\to\infty$, which means that
$\b{\jp{x}^{-1/2+l}\sigma(X,D)u_n}_{n\in\N}$ is a Cauchy sequence in
$H^{s,\tilde{s}}(\R_t\times\R^n_x)$.
Hence we have a limit
\begin{equation}\label{cauchy}
\jp{x}^{-1/2+l}\sigma(X,D)u_n\to\jp{x}^{-1/2+l}w
\end{equation}
in $H^{s,\tilde{s}}(\R_t\times\R^n_x)$ by the completeness of it.
We note that $\n{\jp{x}^{-1/2+l}w}_{H^{s,\tilde{s}}(\R_t\times\R^n_x)}$
is sufficiently small again.
Let $u$ be the solution to 
\[
\left\{
\begin{aligned}
\p{i\partial_t-L}\,u(t,x)=&w^N\\
u(0,x)=\varphi(x),\quad &t\in\R,\, x\in\R^n.
\end{aligned}
\right.
\]
We note that
$\jp{x}^{1/2+l}w^N\in H^{s,\tilde s}(\R_t\times\R^n)$
by \eqref{alg1},
hence 
$w^N\in C^0(\R_t\,;H^{\tilde s}(\R^n_x))$
by Sobolev's embedding in the variable $t$.
Furthermore, on account of the expression
\[
u(t,x)=e^{-itL}\varphi(x)+\frac1i\int_0^te^{-i(t-s)L}w(s,x)^N\,ds,
\]
we have
$u(t,x)\in C^0(\R_t\,;H^{\tilde s}(\R^n_x))\subset C^0(\R_t\times\R^n_x)$
by Sobolev's embedding again.
Then, by Corollary \ref{est:sm2}, \eqref{alg2}, and \eqref{cauchy},
we have
\begin{align*}
&\n{\jp{x}^{-1/2+l}\sigma(X,D) (u_n-u)}_{H^{s,\tilde{s}}(\R_t\times\R^n_x)}
\\
\leq &C
\n{\jp{x}^{1/2+l}(\p{\sigma(X,D)u_{n-1}}^N-w^N)}
_{H^{s,\tilde{s}}(\R_t\times\R^n_x)}
\to0
\end{align*}
hence we have $w=\sigma(X,D)u$ by \eqref{cauchy} again.
This means that $u$ is the desired
time global solution to \eqref{eq:eq4} and
the proof of Theorem \ref{thm:thm7} is complete.


\end{document}